\documentclass[11pt,reqno]{amsart}
\usepackage{amsmath, amsfonts, color, amsthm, graphics, amssymb,url}
\usepackage{stmaryrd}
\usepackage[notcite,notref,final]{showkeys}

\usepackage{cleveref}

\crefname{section}{Section}{Sections}
\crefformat{section}{#2Section~#1#3} 
\Crefformat{section}{#2Section~#1#3} 

\crefname{subsection}{}{Subsections}
\crefformat{subsection}{\S#2#1#3} 
\Crefformat{subsection}{\S#2#1#3}

\crefname{definition}{Definition}{Definitions}
\crefformat{definition}{#2Definition~#1#3} 
\Crefformat{definition}{#2Definition~#1#3} 

\crefname{example}{Example}{Examples}
\crefformat{example}{#2Example~#1#3} 
\Crefformat{example}{#2Example~#1#3} 

\crefname{examplenodiamond}{Example}{Examples}
\crefformat{examplenodiamond}{#2Example~#1#3} 
\Crefformat{examplenodiamond}{#2Example~#1#3} 

\crefname{remark}{Remark}{Remarks}
\crefformat{remark}{#2Remark~#1#3} 
\Crefformat{remark}{#2Remark~#1#3} 

\crefname{remarknodiamond}{Remark}{Remarks}
\crefformat{remarknodiamond}{#2Remark~#1#3} 
\Crefformat{remarknodiamond}{#2Remark~#1#3} 

\crefname{convention}{Convention}{Conventions}
\crefformat{convention}{#2Convention~#1#3} 
\Crefformat{convention}{#2Convention~#1#3} 

\crefname{lemma}{Lemma}{Lemmas}
\crefformat{lemma}{#2Lemma~#1#3} 
\Crefformat{lemma}{#2Lemma~#1#3} 

\crefname{definition-lemma}{Definition-Lemma}{Definition-Lemmas}
\crefformat{definition-lemma}{#2Definition-Lemma~#1#3} 
\Crefformat{definition-lemma}{#2Definition-Lemma~#1#3} 

\crefname{proposition}{Proposition}{Propositions}
\crefformat{proposition}{#2Proposition~#1#3} 
\Crefformat{proposition}{#2Proposition~#1#3} 

\crefname{corollary}{Corollary}{Corollaries}
\crefformat{corollary}{#2Corollary~#1#3} 
\Crefformat{corollary}{#2Corollary~#1#3} 

\crefname{theorem}{Theorem}{Theorems}
\crefformat{theorem}{#2Theorem~#1#3} 
\Crefformat{theorem}{#2Theorem~#1#3} 

\crefname{assumption}{Assumption}{Assumptions}
\crefformat{assumption}{#2Assumption~#1#3} 
\Crefformat{assumption}{#2Assumption~#1#3}

\crefname{notation}{Notation}{Notation}
\crefformat{notation}{#2Notation~#1#3} 
\Crefformat{notation}{#2Notation~#1#3}

\crefname{hypothesis}{Hypothesis}{Hypotheses}
\crefformat{hypothesis}{#2Hypothesis~#1#3} 
\Crefformat{hypothesis}{#2Hypothesis~#1#3}

\crefname{equation}{}{}
\crefformat{equation}{(#2#1#3)} 
\Crefformat{equation}{(#2#1#3)}

\crefname{align}{}{}
\crefformat{align}{(#2#1#3)} 
\Crefformat{align}{(#2#1#3)}

\crefname{proofstep}{Step}{Steps}
\crefformat{proofstep}{#2Step~#1#3} 
\Crefformat{proofstep}{#2Step~#1#3}

\usepackage{tikz}
\usetikzlibrary{arrows,calc,decorations.pathreplacing,decorations.markings,intersections,shapes.geometric,through,fit,shapes.symbols,positioning,decorations.pathmorphing}

\numberwithin{equation}{section}
 \theoremstyle{plain}
\newtheorem{theorem}[equation]{Theorem}

\newtheorem{definition-prop}[equation]{Definition-Proposition}
\newtheorem{lemma}[equation]{Lemma}

\newtheorem{corollary}[equation]{Corollary}
\newtheorem{proposition}[equation]{Proposition}

\theoremstyle{definition}
\newtheorem{definition}[equation]{Definition}

\newtheorem{remark}[equation]{Remark}

\newtheorem{example}[equation]{Example}

\newcommand{\kk}{\Bbbk}

\newcommand{\red}{\textcolor{black}}

\newcommand\bC{\mathbb C}

\newcommand\bE{\mathbb E}
\newcommand\bF{\mathbb F}

\newcommand\bI{\mathbb I}

\newcommand\bN{\mathbb N}

\newcommand\bZ{\mathbb Z}

\newcommand\cB{\mathcal B}

\newcommand\cG{\mathcal G}

\newcommand\cO{\mathcal O}

\def\gk{{\rm GKdim\hspace{.02in} }}

\input xy
\xyoption{all}

\begin{document}

\title[GK-dimension of cosemisimple Hopf algebras]
{Gelfand-Kirillov dimension of cosemisimple Hopf algebras}

\author{Alexandru Chirvasitu}
\address{Department of Mathematics, University at Buffalo
Buffalo, NY 14260, USA}
\email{achirvas@buffalo.edu}

\author{Chelsea Walton}
\address{Department of Mathematics, The University of Illinois at Urbana-Champaign, Urbana, IL 61801,~USA}
\email{notlaw@illinois.edu}

\author{Xingting Wang}
\address{Department of Mathematics, Howard University, Washington, DC 20059, USA}
\email{\red{xingting.wang@howard.edu}}

\begin{abstract}
In this note, we compute the Gelfand-Kirillov dimension of cosemisimple Hopf algebras that arise as deformations of a linearly reductive algebraic group. Our work lies in a purely algebraic setting and generalizes results  of Goodearl-Zhang (2007), of Banica-Vergnioux (2009), and  of D'Andrea-Pinzari-Rossi (2017).
\end{abstract}
 


 \maketitle
 
 

\section{Introduction}\label{se.intro}

Let $\kk$ be an algebraically closed field and all algebraic structures in this note will be $\kk$-linear. Recall from \cite{kl} that the {\it Gelfand-Kirillov (GK-)dimension} of a finitely generated, unital algebra $A$ is the growth measure defined by
  \begin{equation*}
 \gk A =   \limsup_{n \to \infty} \frac{\log (\dim V^n)}{\log n},
  \end{equation*}
 where $V$ is a finite-dimensional generating subspace of $A$ containing $1_A$; this definition is independent of the choice of $V$ \cite[page~14]{kl}.
 
 \smallskip

 The typical techniques for computing the GK-dimension of an algebra involve \red{Gr\"{o}bner basis} methods or some other type of algebraic or representation-theoretic approach. But for the case that we examine here, i.e., when the algebra admits a well-behaved coalgebra structure, we can compute its GK-dimension using corepresentation-theoretic methods instead. Namely, given a finitely generated, cosemisimple Hopf algebra $H$ we consider the invariant $(R_+(H), d_H)$ used in~\cite{bv}; here, $R_+(H)$ is the Grothendieck semiring for the category of finite-dimensional $H$-comodules and $d_H$ is the dimension function on $R_+(H)$. We first establish that if $H$ is finitely generated, then its GK-dimension depends only on $(R_+(H), d_H)$ [Proposition~\ref{prop.dep}].  Then, our main result verifies that if $H$ is a deformation of a linearly reductive group $G$ in the sense that there is an isomorphism \red{between the two} pairs $(R_+(H), d_H)$ and $(R_+(G), d_G)$ [Definition~\ref{def.grng}, Remark~\ref{re.og}], then the GK-dimension of $H$ equals the dimension of $G$ as an algebraic variety [Theorem~\ref{th.main}].

\smallskip

Our results are related to the fact if $\phi: H \to H'$ is a morphism of cosemisimple Hopf algebras that induces a semiring isomorphism $\phi_+: R_+(H) \overset{\sim}{\longrightarrow}  R_+(H')$, then $\phi$ is a Hopf algebra isomorphism; see, e.g.,  \cite[Lemma~5.1]{Bichon:B(E)}. But we do not require such a map $\phi$ here to get $\gk H = \gk H'$; instead we just \red{require a map} $R_+(H) \overset{\sim}{\longrightarrow}  R_+(H')$ that preserves~dimension.

\smallskip

We introduce necessary terminology and verify the results mentioned above in Section~\ref{se.main}. Then, we discuss in Section~\ref{se.examples} how our main theorem compares to previous results on the growth of cosemisimple Hopf algebras  by Goodearl-Zhang \cite{gz}, by Banica-Vergnioux  \cite{bv}, and by D'Andrea-Pinzari-Rossi \cite{apr}. We also compute in Section~\ref{se.examples} the GK-dimensions of Dubois-Violette and Launer's and Mrozinski's universal  quantum groups \cite{dvl, mroz} using Theorem~\ref{th.main}.


\section{Terminology and Main Result}\label{se.main}

Throughout this section, let $H$ be a cosemisimple Hopf algebra. We write $\widehat{H}$ for the set of isomorphism classes $\{[S_\alpha]\}_\alpha$ of simple $H$-comodules $S_\alpha$; sometimes we abuse notation and write $\widehat{H}$ for the corresponding index set $\{\alpha\}$. Note that $\{[S_\alpha]\}_\alpha$ is a basis for the Grothendieck ring $R(H)$ as a free abelian group. The trivial $H$-comodule will be denoted by ${\bf 1}$.  We also write $\le_\oplus$ to denote \red{``direct summand of"}.
Now we discuss  the invariant $(R_+(H), d_H)$ mentioned in the Introduction.

\begin{definition}\label{def.grng}
  The {\it measured Grothendieck semiring} of $H$ is the pair $(R_+(H),d_H)$ where
  \begin{itemize}
  \item $R_+(H)$ is the Grothendieck semiring of the category of finite-dimensional $H$-comodules; and
  \item $d_H: R_+(H)\to \bN$ is the $\kk$-vector space dimension function on $R_+(H)$.    
  \end{itemize}
  We write $(R_+, d)$ if $H$ is understood. 
  
 If $H'$ is another cosemisimple Hopf algebra, then $(R_+(H),d_H)$ is {\it isomorphic} to $(R_+(H'),d_{H'})$ if there is a semiring isomorphism $f: R_+(H) \to R_+(H')$ so that $d_H([S_\alpha]) = d_{H'}(f([S_\alpha]))$ for all $[S_\alpha] \in \widehat{H}$.
\end{definition}

\begin{remark} \label{re.og} Recall that the Hopf algebra $\mathcal{O}(G)$ of regular functions on an affine algebraic group $G$ is cosemisimple if and only if $G$ is linearly reductive \cite[Section~4.6]{abe} which in turn is equivalent to $G$ being reductive in characteristic $0$ (see e.g. \cite[Appendix to Chapter 1]{fkm} and the references therein). We write $(R_+(G), d_G)$ for the measured Grothendieck semiring of $\mathcal{O}(G)$.
\end{remark}

Being cosemisimple, $H$ decomposes as the direct sum of its simple subcoalgebras. Because we are working over an algebraically closed field, the latter are matrix coalgebras necessarily of the form $M_\ell(\kk)^*$ for various ranks $0<\ell<\infty$. In short, we obtain a decomposition 
\begin{equation} \label{eq:decomp}
  H = \bigoplus_{\alpha\in \widehat{H}}C_{\alpha}, \quad \text{for } C_{\alpha}\cong(M_{\dim S_{\alpha}}(\kk))^*.
\end{equation}

Here, there are no multiple copies of any simple subcoalgebra in $H$ in the decomposition above; see, e.g., \cite[Theorem~11.2.13(v)]{ks}. If one chooses a  subspace $V$ of $H$ of the form
\begin{equation}\label{eq:d-alphas}
  V=\bigoplus_{\alpha\in F}C_{\alpha},\quad \text{for some finite subindex set } F \subseteq \widehat{H},  
\end{equation}
then $V^n\subseteq H$ has the following representation-theoretic characterization. 

\begin{lemma}\label{le.dn}

Take $V$ to be a subspace of $H$ as in \eqref{eq:d-alphas}. Then, 
  \begin{equation*}
    V^n = \bigoplus_{\beta \in F_n}C_{\beta}, \quad \text{where }
    F_n :=\left.\Bigg\{\beta\in \widehat{H}\ \right |\ S_{\beta}\le_{\oplus} \left(\bigoplus_{\alpha\in F}S_{\alpha}\right)^{\otimes n}\Bigg\}. 
  \end{equation*}
\end{lemma}

\begin{proof}
Our goal is to show that $V^n$ and $\left(\bigoplus_{\alpha \in F} S_\alpha\right)^{\otimes n}$ have the same simple $H$-comodule direct summands; we do so by showing that they are the same as those for the $H$-comodule $V^{\otimes n}$.

\smallskip

Towards the goal, we recall the notion of the support of an $H$-comodule. For any finite dimensional $H$-comodule $M$ with structure map $\rho: M\to M\otimes H$, we denote by $C(M)$ the unique minimal subspace of $H$ satisfying $\rho(M)\subset M\otimes C(M)$; we call $C=C(M)$ the {\it support} of the $H$-comodule $M$. By \cite[Theorem 3.2.11(c)]{Radford}, $C$ is a finite-dimensional subcoalgebra of $H$. Since $H$ is cosemisimple, so is $C$, and hence, $C^*$ is semisimple. This implies that $M$, as a module over $H^*$, is a direct sum of the simple modules appearing as direct summands of $H^*/\text{Ann}_{H^*}(M) \cong C^*$. By \cite[Corollary~3.2.6]{Radford}, there is a one-to-one correspondence between $C^*$-submodules of $M$ and $C$-subcomodules of $M$.  So, $M$ is the \red{direct sum of simple} $C$-comodules, i.e, $M$ is \red{isomorphic to} the coradical $C(M)_0$ of its support $C=C(M)$. Therefore, two $H$-comodules with the same support must have the same simple $H$-comodule direct summands.

\smallskip Now it suffices to show that the $H$-comodules $V$ and $\bigoplus_{\alpha \in F}  S_\alpha$, and the $H$-comodules $V^{\otimes n}$ and $V^n$, have the same support. The first statement holds by \eqref{eq:d-alphas} as $C_\alpha \cong (S_\alpha)^{\oplus \dim S_\alpha}$ as $H$-comodules. For the second statement, note the multiplication map $\mu_n:V^{\otimes n} \to V^n$ is a surjective $H$-comodule map, which induces an injective coalgebra map $C(V^n) \hookrightarrow C(V^{\otimes n})$. Since $V^n$ is a subcoalgebra of $H$, we have $C(V^n)=V^n$. Moreover, the $H$-comodule structure on $V^{\otimes n}$ is given by \red{the restriction of}
\[
\xymatrix{
H^{\otimes n}\ar[r]^-{\Delta_n} &H^{\otimes n}\otimes H^{\otimes n}\ar[r]^-{1\otimes \mu_n} &H^{\otimes n}\otimes H.
}
\]
Since $\Delta(V)\subset V\otimes V$, so we get $C(V^{\otimes n})\subseteq \mu_n(V^{\otimes n})=V^n$. Therefore, we have
$$C(V^{\otimes n})\subseteq \mu_n(V^{\otimes n})=V^n=C(V^n)\subseteq C(V^{\otimes n}).$$
This implies $V^{\otimes n}$ and $V^n$ share the same support.
\end{proof}

The result below is an immediate consequence of Lemma~\ref{le.dn} and \eqref{eq:decomp}. 

\begin{corollary}\label{cor.dim-dn}
Retaining the notation of Lemma~\ref{le.dn}.  For representatives $S_\beta$ of $[S_\beta] \in \widehat{H}$, we have that
  \begin{equation*}
    \dim V^n = \sum_{\beta\in F_n} (\dim S_{\beta})^2. 
  \end{equation*}

  \vspace{-.25in} 
  
  \qed
\end{corollary}

Now we obtain the following result.

\begin{proposition}\label{prop.dep}
  The Gelfand-Kirillov dimension of a finitely generated, cosemisimple Hopf algebra $H$ only depends on its measured Grothendieck semiring. 
\end{proposition}

\begin{proof}
By the ``local finiteness" property of coalgebras \cite[Theorem 5.1.1]{MO93}, we can always choose a finite-dimensional subcoalgebra of $H$ so that it contains a \red{finite-dimensional} vector space $V$ that generates $H$ as an algebra and that $1\in V$ (i.e., the trivial $H$-comodule ${\bf 1}$ is a direct summand of $V$). Further, we can take $V$ of the form \eqref{eq:d-alphas}.

\smallskip

Note that the simple $H$-comodule \red{index set} $F_n$ of $V^n$ in Lemma~\ref{le.dn} is uniquely determined by the Grothendieck semiring structure of $H$; indeed, $\beta$ belongs to $F_n$ if and only if
  \begin{equation*}
    \left(\sum_{\alpha\in F}[S_{\alpha}]\right)^n= [S_{\beta}]+x
  \end{equation*}
  for some $x\in R_+$. The conclusion now follows from \Cref{cor.dim-dn} and the definition of GK-dimension since $\dim V^n$ can be calculated by using the measure $d_H$ on all of the $[S_\beta]$'s  that belong to $F_n$.
\end{proof}

We now apply the discussion above to Hopf algebras obtained as certain deformations of the Hopf algebra $\mathcal{O}(G)$ in Remark~\ref{re.og}.

\begin{definition}\label{def.qtm}
  Let $G$ be a linearly reductive algebraic group. A cosemisimple Hopf algebra $H$ is said to be a $G$-{\it deformation} if $R_+(H) \cong R_+(G)$. If, further, $(R_+(H), d_H) \cong (R_+(G), d_G)$, \red{then we will say that $H$ is a} {\it quantum function algebra on $G$}.
  \end{definition}
  
The main result of this note is given below.

\begin{theorem}\label{th.main}
  Let $G$ be a linearly  \red{reductive affine algebraic}  group and let $H$ be  a finitely generated,  quantum function algebra on $G$ in the sense of \Cref{def.qtm}. Then, we have that
  \begin{equation*}
    \gk H = \dim G,
  \end{equation*}
  where the latter is the dimension of $G$ as an algebraic variety. 
\end{theorem}

\begin{proof}
By \Cref{prop.dep}, it is enough to prove this for $H=\cO(G)$ itself. Recall from \cite[Theorem~4.5(a)]{kl} that the GK-dimension of a finitely generated, commutative algebra coincides with its classical Krull dimension. Moreover, $\cO(G)$ is finitely generated as $G$ is affine, and the classical Krull dimension of an algebra of regular functions on an algebraic variety is simply the dimension of that variety.
\end{proof}


\section{Examples and Previous Results}\label{se.examples}

In this section, we highlight special cases of Theorem~\ref{th.main}. We specialize to 
the case when $\kk$ has characteristic 0 throughout the section, since the requirement that an algebraic group be linearly reductive is rather strong in positive characteristic. Indeed, in characteristic $p>0$ the only linearly reductive affine algebraic groups are extensions of tori by finite groups of order coprime to~$p$; see~\cite{nag}.

\smallskip

The next few results take place in the same general setting, which we now recall briefly. 
Take:
\begin{itemize}
\item $G$,  a semisimple, connected, simply connected algebraic group; 
\item $\mathfrak{g}$, the Lie algebra of $G$; and
\item  $q$ \red{a nonzero scalar}. If $q$ is a root of unity of order $\ell$, then we assume that $\ell$ is odd, and  coprime to 3 when $\mathfrak{g}$ contains a $G_2$-component.
\end{itemize}
Recall that we can define a $q$-deformed version $U_q(\mathfrak{g})$ of the enveloping algebra $U(\mathfrak{g})$; see, e.g.,  \cite[Chapter~9]{cp}.
We refer the reader to \cite{dc-ly} and \cite[Sections~I.7 and~III.7]{bg} for \red{details on how to then define the Hopf subalgebra $\cO_q(G)$ of the Hopf dual $U_q(\mathfrak{g})^{\circ}$ in the $q$ a root of unity and in the $q$ generic cases}.

\begin{proposition} \label{prop.oqg}
Retain the hypotheses above. Then,  $\gk \cO_q(G) = \dim G$. 
\end{proposition}

\begin{proof}

First, assume that $q$ is not a root of unity. 
Then, \cite[Theorem 10.1.14]{cp} implies that $\cO_q(G)$ is cosemisimple. Moreover, \cite[Proposition 10.1.16]{cp}  implies that $\cO_q(G)$ is a quantum function algebra on $G$. Thus, \Cref{th.main}  applies to obtain that $\gk \cO_q(G) = \dim G$.

\smallskip

Now suppose $q$ is a root of unity of order $\ell$ as above. The $\ell =1$ case holds by Remark~\ref{re.og}. Next, note that when $\ell>1$ is odd, we have that $\cO_q(G)$ is a finite module over a  subalgebra isomorphic to $\cO(G)$, by \cite[Proposition~6.4 and Theorem~7.2]{dc-ly} (see also \cite[Section~III.7.2]{bg}). So, by \cite[Proposition~5.5]{kl}, $\gk \cO_q(G) = \gk \cO(G)$, which is equal to $\dim G$ \red{as in the proof of Theorem~\ref{th.main}}.
\end{proof}

\begin{remark}\label{rmk.gz}
The case when  $q\in \bC^{\times}$ is transcendental over $\mathbb{Q}$ recovers a result of Goodearl-Zhang~\cite{gz}. Namely, \cite[Theorem~0.1]{gz} shows that $\cO_q(G)$ is Auslander-regular and Cohen-Macaulay, and its GK-dimension is as in \Cref{th.main} (and also equal to its global dimension).  
\end{remark}

The next three remarks pertain to connected, simply connected, compact real Lie groups $G$.

\begin{remark}\label{rmk.bv}
A weight-theory-based argument appears in work of Banica-Vergnioux \cite{bv} for a particular case of \Cref{th.main}: Namely, \cite[Theorem~2.1]{bv} computes the GK-dimension for algebras of regular functions on connected, simply connected, compact, real Lie groups. This is done in unitary language, working with the maximal compact subgroups of such linear algebraic groups.
\end{remark}

\begin{remark}\label{rmk.apr}
The result in \Cref{rmk.bv} was then extended by different methods,  close in spirit to what we achieve here, to representative functions on arbitrary compact groups (i.e. regular functions on classical reductive groups) in work of D'Andrea-Pinzari-Rossi \cite[Corollary 3.5]{apr}. 
\end{remark}

\begin{remark}
In the real, unitary setting, Proposition~\ref{prop.oqg} also recovers the main result of \cite{cs}. This is proved under the additional assumption that $q$ is positive and $<1$ in order to make use of an appropriate $*$-structure on $\cO_q(G)$. 
\end{remark}

We also emphasize that Theorem~\ref{th.main} can be applied to $G$-deformations more general that those arising in Proposition~\ref{prop.oqg}, such as cosemisimple multi-parameter deformations of $\mathcal{O}(G)$. For instance, in the cosemisimple case, Takeuchi's two-parameter deformations of $GL(2)$ \cite{tak} are a subclass of the quantum groups discussed in Example~\ref{ex.mroz} below.

\smallskip

Now we turn our attention to the growth of the $SL(2)$-deformations and the $GL(2)$-deformations studied in \cite{Bichon:B(E)} and \cite{mroz}, respectively. To begin, we need the result below.

\begin{lemma}\label{le.dim-ineq}
Let $H$ and $K$ be two finitely generated, cosemisimple Hopf algebras, and suppose that there exists an isomorphism  $f: R_+(H) \to R_+(K)$ between their Grothendieck semirings. If there is a class $[X]\in R_+(H)$ such that $\dim X>\dim f(X)$, then $\gk H = \infty$. 
\end{lemma}

\begin{proof}
Since $X^{\otimes n}$ and $f(X)^{\otimes n}$ have the same number of simple factors, we get that
$$\text{length}\hspace{.02in}X^{\otimes n} \; = \; \text{length}\hspace{.02in}f(X)^{\otimes n} \; \leq \dim f(X)^{\otimes n} \; = \; (\dim f(X))^n.$$
On the other hand, there is a simple $H$-comodule $S_n$, which is a direct summand $X^{\otimes n}$, with 
$$(\text{length}\hspace{.02in} X^{\otimes n})(\dim S_n) \; \geq \; \dim X^{\otimes n} \; = \; (\dim X)^{n}.$$ 
Hence, $\dim S_n\ge (\dim X/\dim f(X))^n$. Now we have that 
$$\gk H \; \ge \; \limsup_{n \to \infty} \frac{\log (\dim X^n)}{\log n} \; \overset{\textnormal{Cor.}~\ref{cor.dim-dn}}{\ge}  \;\limsup_{n \to \infty} \frac{\log (\dim S_n)^2}{\log n} \; \ge \; \limsup_{n \to \infty} \frac{2n}{\log n}\log\left(\frac{\dim X}{\dim f(X)}\right),$$
and from the hypothesis that $\dim X>\dim f(X)$ we obtain the desired result. 
\end{proof}

\begin{remark}
\Cref{le.dim-ineq} is analogous to both \cite[Proposition 2.8]{den} and \cite[Proposition~6.1]{ban-sbf}.
The latter is phrased in an analytic setting for compact quantum groups satisfying an amenability condition that is, in general, weaker than polynomial or even sub-exponential growth (see e.g. \cite[Theorem 4.6]{apr}).  On the other hand,  \cite[Proposition 2.8]{den} is phrased in terms of a monoidal functor between categories of comodules over Hopf algebras (that are not necessarily cosemisimple) rather than for morphisms of Grothendieck semirings.
\end{remark}

  \begin{lemma}\label{le.dim-eq}
  \red{  Let $q\in \bC^{\times}$ be a generic scalar. A semiring homomorphism $f$ from the Grothendieck semiring $R_+(\cO_q(SL(2)))$ or $R_+(\cO_q(GL(2)))$ to $\mathbb Z_+$ for generic $q$ is uniquely determined by $f(V)$, where $V$ is any $2$-dimensional simple comodule.}
\end{lemma}

\begin{proof}
 \red{We prove the claim for $\cO_q(GL(2))$, the other case being analogous.}

  \red{The morphism $f$ lifts to one of {\it rings} (rather than semirings)
  \begin{equation}\label{eq:1}
    R(\cO_q(GL(2)))\to \bZ.
  \end{equation}
  The left hand ring is isomorphic to $\bZ[x,y^{\pm 1}]$ where $x$ is the class of any simple $2$-dimensional comodule while $y$ is the class of a $1$-dimensional comodule (this follows from instance from \cite[Chapter 14, Theorem 3.1]{huse} and its proof).}

 \red{ Since $y$ is mapped to an invertible positive integer, we have $f(y)=1$. The morphism \Cref{eq:1} thus factors through $\bZ[x]$ and is uniquely determined by the image of $x$, as claimed.}
\end{proof}


\begin{example}\label{ex.bic-dv}
Let $V$ \red{be a $d$-dimensional} vector space and  $\bE \in GL_d(\mathbb{C})$ \red{a matrix} encoding a bilinear form on $V$, \red{for $d \geq 2$}. The quantum automorphism group $\cB(\bE)$ was introduced in \cite{dvl} 
and its comodule theory is studied in \cite{Bichon:B(E)}.  

\smallskip
It is shown in \cite[Theorem 1.2]{Bichon:B(E)} that each $SL(2)$-deformation [Definition~\ref{def.qtm}] is isomorphic to $\cB(\bE)$, for some $\bE \in GL_d(\bC)$ such that the solution to $q^2 + \text{tr}(\bE^T \bE^{-1})q+1 =0$ is {\it generic}, that is, $q$ equal to $\pm 1$ or a non-root of unity. 
In fact, the quantized coordinate ring $\cO_q(SL(2))$  is cosemisimple if and only if  $q$ is generic; see Remark~\ref{re.og}, \cite[Section~4.2.5]{ks}, and the discussion in \cite[Section~5]{Bichon:B(E)}. According to \cite[Theorem~1.1]{Bichon:B(E)}, the category of $\cB(\bE)$-comodules is equivalent to that of $\cO_q(SL(2))$  as monoidal categories (i.e., there exists a {\it monoidal Morita-Takeuchi equivalence}) for $q\in \bC^{\times}$ satisfying $q^2 + \text{tr}(\bE^T \bE^{-1})q+1 =0$. 

\smallskip

Restricting our attention to the case when $\cB(\bE)$ is cosemisimple (or, when it is monoidally Morita-Takeuchi equivalent to $\cO_q(SL(2))$ for $q$ generic) the proofs of \red{the results} mentioned above \red{make it clear} that   $V$ maps to the fundamental $2$-dimensional $\cO_q(SL(2))$-comodule under a semiring isomorphism, $$f: R_+(\cB(\bE)) \overset{\sim}{\longrightarrow} R_+(\cO_q(SL(2))).$$
Since $\dim V = d$ and $\dim f(V) = 2$, we obtain  
 that $\gk  \cB(\bE) =  \infty$ for $d >2$ by Lemma~\ref{le.dim-ineq}. On the other hand, \red{by \cite[Lemma 5.2]{Bichon:B(E)} any monoidally Morita-Takeuchi equivalence functor from $\cO_q(SL(2))$ to $\cB(\bE)$ must send the $2$-dimensional simple $\cO_q(SL(2))$-comodule to the fundamental simple $\cB(\bE)$-comodule $V$ of dimension $d$. Hence we get that $\gk \cB(\bE) = \dim SL(2) = 3$ when $d=2$ by Lemma~\ref{le.dim-eq} and Theorem~\ref{th.main}}. 
 
  \smallskip
This extends, in the case when $\cB(\bE)$ is cosemisimple,  the result that $\gk \cB(\bE) < \infty $ if and only if $d=2$, obtained as part of \cite[Theorem 0.3]{ww}.
\end{example}

\begin{example}\label{ex.mroz}
 Similarly, there are $GL(2)$-deformations $\cG(\bE,\bF)$  introduced and studied in \cite{mroz} that are defined by  $\bE, \bF \in GL_d(\mathbb{C})$ so that $\bF^T\bE^T\bE \bF = \lambda \bI$ for $\lambda \in \mathbb{C}^\times$. We have that $\cG(\bE,\bF)$ and $\cO_q(GL(2)))$ are monoidally Morita-Takeuchi equivalent for $q\in \bC^{\times}$ satisfying $$q^2 - \sqrt{\lambda^{-1}} \text{tr}(\bE^T \bE^{-1})q+1 =0$$ \cite[Theorem~1.1]{mroz}. If, further, $q$ is generic then $R_+(\cG(\bE,\bF))\cong R_+(\cO_q(GL(2)))$ \cite[Theorem~1.2]{mroz}.  Recall that $\cO_q(GL(2))$  is cosemisimple if and only if  $q$ is generic; see Remark~\ref{re.og}, \cite[Section~11.5.4]{ks}, and the discussion in \cite[Section~4]{mroz}.
  
 \smallskip
 Now restricting our attention to the case when $\cG(\bE,\bF)$ is cosemisimple, $\gk  \cG(\bE,\bF) =  \infty$ for $d >2$ by Lemma~\ref{le.dim-ineq}. \red{Moreover, by  \cite[Lemma 4.1]{mroz} any monoidally Morita-Takeuchi equivalence functor $F:  \cO_q(GL(2))\to \cG(\bE,\bF)$ sends some $2$-dimensional simple $\cO_q(GL(2))$-comodule to the fundamental simple $\cG(\bE,\bF)$-comodule of dimension $d$. Hence $\gk \cG(\bE,\bF) = \dim GL(2) = 4$ when $d=2$ by Lemma~\ref{le.dim-eq} and Theorem~\ref{th.main}}.  
  \end{example}


\bigskip

\subsection*{Acknowledgments}
\red{We thank the anonymous referee for their comments that improved the expositions of this manuscript.} We would \red{also} like to thank Ken Brown, Ken Goodearl, and James Zhang for useful correspondences  and for providing references on this subject. A. Chirvasitu and C. Walton are partially supported  by the US National Science Foundation with grants  \#DMS-1801011 and \#DMS-1663775, respectively.  C. Walton is also supported by a research fellowship from the Alfred P. Sloan foundation.


\bibliographystyle{abbrv}   

\begin{thebibliography}{10}

\bibitem{abe}
E.~Abe.
\newblock {\em Hopf algebras}, volume~74 of {\em Cambridge Tracts in
  Mathematics}.
\newblock Cambridge University Press, Cambridge-New York, 1980.
\newblock Translated from the Japanese by Hisae Kinoshita and Hiroko Tanaka.

\bibitem{ban-sbf}
T.~Banica.
\newblock Representations of compact quantum groups and subfactors.
\newblock {\em J. Reine Angew. Math.}, 509:167--198, 1999.

\bibitem{bv}
T.~Banica and R.~Vergnioux.
\newblock Growth estimates for discrete quantum groups.
\newblock {\em Infin. Dimens. Anal. Quantum Probab. Relat. Top.},
  12(2):321--340, 2009.

\bibitem{Bichon:B(E)}
J.~Bichon.
\newblock The representation category of the quantum group of a non-degenerate
  bilinear form.
\newblock {\em Comm. Algebra}, 31(10):4831--4851, 2003.

\bibitem{bg}
K.~A. Brown and K.~R. Goodearl.
\newblock {\em Lectures on algebraic quantum groups}.
\newblock Advanced Courses in Mathematics. CRM Barcelona. Birkh\"auser Verlag,
  Basel, 2002.

\bibitem{cs}
P.~S. {Chakraborty} and B.~{Saurabh}.
\newblock {Gelfand-Kirillov dimension of the algebra of regular functions on
  quantum groups. Preprint available at \url{http://arxiv.org/pdf/1709.09540}}.
\newblock 2017.

\bibitem{cp}
V.~Chari and A.~Pressley.
\newblock {\em A guide to quantum groups}.
\newblock Cambridge University Press, Cambridge, 1995.
\newblock Corrected reprint of the 1994 original.

\bibitem{apr}
A.~D'Andrea, C.~Pinzari, and S.~Rossi.
\newblock Polynomial growth of discrete quantum groups, topological dimension
  of the dual and {$*$}-regularity of the {F}ourier algebra.
\newblock {\em Ann. Inst. Fourier (Grenoble)}, 67(5):2003--2027, 2017.

\bibitem{den}
A.~{Davydov}, P.~{Etingof}, and D.~{Nikshych}.
\newblock {Autoequivalences of tensor categories attached to quantum groups at
  roots of $1$. Preprint available at \url{http://arxiv.org/pdf/1703.06543}}.
\newblock 2017.

\bibitem{dc-ly}
C.~De~Concini and V.~Lyubashenko.
\newblock Quantum function algebra at roots of {$1$}.
\newblock {\em Adv. Math.}, 108(2):205--262, 1994.

\bibitem{dvl}
M.~Dubois-Violette and G.~Launer.
\newblock The quantum group of a nondegenerate bilinear form.
\newblock {\em Phys. Lett. B}, 245(2):175--177, 1990.

\bibitem{gz}
K.~R. Goodearl and J.~J. Zhang.
\newblock Homological properties of quantized coordinate rings of semisimple
  groups.
\newblock {\em Proc. Lond. Math. Soc. (3)}, 94(3):647--671, 2007.

\bibitem{huse}
D.~Husemoller.
\newblock {\em Fibre bundles}, volume~20 of {\em Graduate Texts in
  Mathematics}.
\newblock Springer-Verlag, New York, third edition, 1994.

\bibitem{ks}
A.~Klimyk and K.~Schm\"udgen.
\newblock {\em Quantum groups and their representations}.
\newblock Texts and Monographs in Physics. Springer-Verlag, Berlin, 1997.

\bibitem{kl}
G.~R. Krause and T.~H. Lenagan.
\newblock {\em Growth of algebras and {G}elfand-{K}irillov dimension},
  volume~22 of {\em Graduate Studies in Mathematics}.
\newblock American Mathematical Society, Providence, RI, revised edition, 2000.

\bibitem{MO93}
S.~Montgomery.
\newblock {\em Hopf algebras and their actions on rings}, volume~82 of {\em
  CBMS Regional Conference Series in Mathematics}.
\newblock Published for the Conference Board of the Mathematical Sciences,
  Washington, DC; by the American Mathematical Society, Providence, RI, 1993.

\bibitem{mroz}
C.~Mrozinski.
\newblock Quantum groups of {$\rm GL(2)$} representation type.
\newblock {\em J. Noncommut. Geom.}, 8(1):107--140, 2014.

\bibitem{fkm}
D.~Mumford, J.~Fogarty, and F.~Kirwan.
\newblock {\em Geometric invariant theory}, volume~34 of {\em Ergebnisse der
  Mathematik und ihrer Grenzgebiete (2) [Results in Mathematics and Related
  Areas (2)]}.
\newblock Springer-Verlag, Berlin, third edition, 1994.

\bibitem{nag}
M.~Nagata.
\newblock Complete reducibility of rational representations of a matric group.
\newblock {\em J. Math. Kyoto Univ.}, 1:87--99, 1961/1962.

\bibitem{Radford}
D.~E. Radford.
\newblock {\em Hopf algebras}, volume~49 of {\em Series on Knots and
  Everything}.
\newblock World Scientific Publishing Co. Pte. Ltd., Hackensack, NJ, 2012.

\bibitem{tak}
M.~Takeuchi.
\newblock A two-parameter quantization of {${\rm GL}(n)$} (summary).
\newblock {\em Proc. Japan Acad. Ser. A Math. Sci.}, 66(5):112--114, 1990.

\bibitem{ww}
C.~Walton and X.~Wang.
\newblock On quantum groups associated to non-noetherian regular algebras of
  dimension 2.
\newblock {\em {M}ath. {Z}.}, 284(1):543--574, 2016.

\end{thebibliography}

\def\cprime{$'$}

\end{document}